\documentclass[11pt]{article}
    \title{Complete Integrability for Piecewise-Smooth Distributions}
    \author{Jack McKee\thanks{Department of Mathematics, University of Hawai`i at M$\overline{\mbox{a}}$noa, ORCID 0009-0004-1865-9089, jmckee@math.hawaii.edu}    }
    \usepackage[bindingoffset=0.2in,
            left=1in,
            right=1.5in,
            top=1in,
            bottom=1in,
            footskip=.25in]{geometry}
            
    \usepackage{algpseudocodex}
    \usepackage{amsmath}
    \usepackage{amsfonts}
    \usepackage{amsthm}
    \usepackage{amssymb}
    \usepackage{xcolor}
    \usepackage{enumitem}
    \newtheorem{theorem}{Theorem}
    \newtheorem{definition}[theorem]{Definition}
    
    \newtheorem{example}[theorem]{Example}
    
        \newtheorem{proposition}[theorem]{Proposition}

    \newtheorem{lemma}[theorem]{Lemma}
    
    \newcommand{\pdiff}[2]{\frac{\partial #1}{\partial #2}}

\usepackage{graphicx}

\usepackage[style=numeric-comp,giveninits=true,url=false,doi=false,isbn=false,date=year,maxbibnames=100,backend=biber]{biblatex}
\bibliography{bibliography.bib}

\usepackage{hyperref}

\begin{document}
	
	\maketitle 
	\begin{abstract}
		Two generalizations of the Frobenius integrability theorem are proved concerning distributions which are piecewise-$C^1$ but may fail to be continuous. The criteria presented are sufficient, but not necessary, for complete integrability of such distributions with bi-Lipschitz coordinates.
	\end{abstract}

	\section{Introduction}
	
	A (smooth) distribution $D$ of corank $k$ on a manifold $M$ is a smooth sub-bundle of the tangent bundle, that is, the data of a subspace $D_x \subset T_xM$ for each point $x \in M$, such that $D_x$ has constant dimension $n - k$ and the map $x \mapsto D_x$ is smooth. An integral manifold of a distribution is an embedded submanifold $S \subseteq M$ such that $T_xS \subset D_x$ for all $x \in S$. In many cases, it is more convenient to work with a corresponding \emph{Pfaffian system} $\alpha = (\alpha^1,\dots,\alpha^k)$, where each $\alpha^i: M \to T^*M$ is a smooth differential one-form and $D_x = \ker\alpha^1|_x \cap \dots \cap \ker\alpha^k|_x$ for each $x$. In this case we say $\alpha$ \emph{annihilates} $D$. 

	A distribution is said to be \emph{completely integrable by $C^r$ diffeomorphisms} at the point $x_0 \in M$ if there exists an open neighborhood $U$ containing $x_0$ and a map $\Phi: U \to \mathbb{R}^n$ which is a $C^r$ diffeomorphism onto its image, and such that $\Phi(x_0) = 0$ and $d\Phi|_x(D) = \mathrm{Span}(\pdiff{}{y^{k+1}},\dots,\pdiff{}{y^n})$ for all $x \in M$. The level sets of these maps fit together to form integral manifolds of the distribution.
		
	A theorem usually attributed to Frobenius establishes necessary and sufficient conditions for complete integrability.
	\begin{theorem}{Frobenius Integrability Theorem \cite{Warner1983}:}\label{classicalfrob}
	
		The smooth distribution annihilated by $(\alpha^1,\dots,\alpha^k)$ is completely integrable by $C^\infty$ diffeomorphisms if and only if for any $x_0 \in M$, there exists an open set $U$ containing $x_0$ and one-forms $A^i_j\in \Lambda^1(U),~ i,j = 1,\dots,k$ such that
		\[d\alpha^i = A^i_j \wedge \alpha^j.\]
		Equivalently, for any pair of vector fields $f_1,f_2 \in \Gamma(D|_U)$, $[f_1,f_2] \in \Gamma(D|_U)$.
		A distribution satisfying this condition is called \emph{involutive}.
	\end{theorem}

	The classical Frobenius theorem applies equally well to $C^r$ distributions and complete integrability by $C^r$ diffeomorphisms, for $r \ge 1$. Various generalizations to less-smooth differential forms exist; of special note is that the classical theorem can be extended to locally Lipschitz differential forms with complete integrability by bi-Lipschitz diffeomorphisms \cite{Simic1996,Rampazzo2007}, in which case the forms $A^i_j$ need only be locally bounded and the requirement that $d\alpha^i = A^i_j \wedge \alpha^j$ is weakened to only be true almost everywhere. The Chow-Rashevskii theorem has also been proved for locally Lipschitz distributions \cite{RampazzoSussman2001}.
	
	The present paper focuses on the following setting: suppose $D = \ker \alpha$ is piecewise-$C^1$ but not necessarily continuous on the smooth manifold $M$. We seek sufficient conditions for the integrability of $D$ by continuous piecewise-$C^1$ diffeomorphisms.

	It is trivial to show that in this setting, involutivity almost everywhere is not sufficient for complete integrability. Even set-involutivity (per the definitions in \cite{Rampazzo2007}) is insufficient. More assumptions need to be added to ensure that the relevant ordinary differential equations used to construct the integral manifolds actually have solutions. Here, the benefit of using the formulation of the Frobenius theorem in terms of differential forms becomes clear, as it is easier to clearly state conditions regarding the annihilator $\alpha$ rather than $D$ itself.

	The question of which conditions should be added is subtle and dependent on the intended application. Theorem \ref{mainthrm2} adds a condition that is quite general and topological in nature, but excludes some simple cases of interest, while Theorem \ref{mainthrm1} adds a condition that is analytical in nature and can be productively applied in some special cases, especially if all of the one-forms in $\alpha$ are piecewise-closed. Both rely on Lemma \ref{mainlemma}, which applies a generalized inverse function theorem due to Clarke \cite{Clarke1976}.

	\section{Notation and Conventions}

	Before stating the theorems of this paper, it will be helpful to fix some notation and conventions.
	
	A \emph{polyhedral mesh} $\mathcal{T}$ on the manifold $M$ (which may have polyhedral boundary) is a set of subsets $\Delta \subset M$ which cover $M$, such that each set $\Delta$ is the image of a closed convex $n-d$-dimensional polytope $\hat\Delta \subset \mathbb{R}^{n-d}$ under a smooth embedding $f_\Delta: \mathbb{R}^{n-d} \to M$. As an abuse of notation, the set $\Delta$ is called a polytope of the mesh, or sometimes just a polytope, of codimension $d$. The image of each face of $\hat\Delta$ is also a polytope of the mesh, which will abusively be called a face of $\Delta$, and the intersection of any two polytopes $\Delta,\Delta'$ of the mesh must be either empty or equal to a face of both. Top-dimensional polytopes of the mesh, or polytopes of codimension zero, will be labeled $T$ to distinguish their importance.

	The interior of a polytope, denoted $\mathring{\Delta}$, is the image of the interior of $\hat{\Delta}$ under $f_\Delta$, where the interior of $\hat{\Delta}$ consists of all points $ta_1 + (1-t)a_2$ with $a_1,a_2 \in \Delta$ and $t \in (0,1)$.
	
	We will always assume that $D$ is the kernel of a Pfaffian system $\alpha$ which is piecewise-$C^1$ with respect to the polyhedral mesh $\mathcal{T}$, i.e. its restriction to each top-dimensional polytope is continuously differentiable. For each top-dimensional polytope $T$, let $\alpha_T$ be $\alpha$ restricted to $T$. A map $\Phi: U \subset M \to \mathbb{R}^n$ is a continuous piecewise-$C^1$ diffeomorphism onto its image if $\Phi$ is a one-to-one open map and $\Phi|_T$ is a $C^1$ embedding for each $T \subset M$. The Pfaffian system $\alpha$ is trivializeable by continuous piecewise-$C^1$ diffeomorphisms at $x_0 \in M$ if there exists an open set $U \ni x_0$ and a map $\Phi:U \to \mathbb{R}^n$ which is a continuous piecewise-$C^1$ diffeomorphism onto its image and $\alpha_T^i = \Phi|_T^*dy^i$ for each $T$ containing $x_0$. The distribution $D = \ker \alpha$ is completely integrable by continuous piecewise-$C^1$ diffeorphisms at $x_0$ if it has a piecewise-$C^1$ annihilator that is trivializeable by continuous piecewise-$C^1$ diffeomorphisms at $x_0$.

	\section{The Main Lemma}

	The next lemma establishes a basic criterion for $\alpha$ to be trivializeable by a continuous piecewise-$C^1$ diffeomorphism.

	\begin{lemma}\label{mainlemma}
		Let $x_0 \in M$ and suppose there exists an open set $U$ containing $x_0$ and a set of maps $\Phi^i: U \to \mathbb{R}$ which are each continuous and piecewise-$C^1$, and $d\Phi^i_T = \alpha^i_T$ for each $i,T$.

		If there exists a subspace $V \subset T_{x_0}M$ such that for each set of nonnegative weights $\{c^T\}_{T \ni x_0}$ satisfying $\sum_T c^T = 1$, $\sum_T c^T \alpha_T|_{x_0}$ has full rank when restricted to $V$, then $\alpha$ is trivializeable by piecewise-$C^1$ diffeomorphisms at $x_0$.
	\end{lemma}
	\begin{proof}
		Let $x_0 \in M$, and  let $f_1,\dots,f_k$ be vectors that span $V$. Let $U$ be a simply connected coordinate neighborhood of $x_0$ which is small enough that every polytope that intersects $U$ contains $x_0$. Using these coordinates, $f_1,\dots,f_k$ can be thought of as fixed elements of $\mathbb{R}^n$.
		Define $\Phi_T(x) := P_{V^\perp}(x) + f_i \Phi^i_T(x)$, where $P_{V^\perp}$ is the orthogonal projection onto $V^\perp$ (in the coordinates on $U$). The proof works by showing that on a small enough open set containing $x_0$, the map $\Phi(x) := \Phi_T(x)$ for all $x \in T$ is bi-Lipschitz, and hence is a piecewise-$C^1$ diffeomorphism onto its image.

		We will use a generalization of the inverse function theorem for Lipschitz maps:
		\begin{theorem}{\cite{Clarke1976}}\label{inversefunctionthrm}			Suppose $f: U \to \mathbb{R}^n$ is a Lipschitz map, and let $df|_{x_0}$ be the convex hull of the values $\lim_{n \to \infty}df|_{x_n}$, for all sequences $(x_n)_{n \in \mathbb{N}}$ that converge to $x_0$ and such that $f$ is differentiable at $x_n$. If every element of $df|_{x_0}$ is invertible, then there exists an open set $U' \subset U$ such that $f$ is bi-Lipschitz on $U'$.
		\end{theorem}
		
		In this case, $d\Phi|_{x_0}$ is the convex hull of $\{d\Phi_T|_{x_0}\}_{T \ni x_0}$. Assume $u$ is a vector in $T_{x_0}M$ and $\sum_{T \ni x_0} c^Td\Phi_T|_{x_0}(u) = 0$ for some set of weights with $\sum_{T \ni x_0} c^T = 1$ and $0 \le c^T \le 1$. Let $u = v + w$, where $v \in V$ and $w \in V^\perp$. Then:
		
		\begin{align*}\sum_{T\ni x_0} c^T[P_{V^\perp} + f_i d\Phi^i_T|_{x_0}] (v + w) &= 0\\
			\sum_{T\ni x_0}c^T[w + f_i\alpha^i_T(v + w)] &= 0 \\
			w + f_i\sum_{T\ni x_0} c^T\alpha^i_T(v+w) &= 0
		\end{align*}
		Since $f_i \in V$ and the vectors $f_1,\dots,f_k$ are linearly independent, the last line is equivalent to $w = 0$ and $\sum_{T \ni x_0} c^T \alpha^i_T(v) = 0$ for all $i$. However, the assumptions of the theorem imply that in this case, $v = 0$. Therefore, $u = 0$.
		By Theorem \ref{inversefunctionthrm}, $\Phi$ is a bi-Lipschitz map when restricted to some open set $U' \subset U$ containing $x_0$. Let $\tilde{\alpha}^i$ be defined by $\tilde{\alpha}^i(f_j) = \delta^i_j$ and $\tilde{\alpha}^i(w) = 0 \forall w \in V^\perp$. Then $\Phi_T^*\tilde{\alpha}^i = d\Phi^i_T$, and $\tilde{\alpha} = (\tilde{\alpha}^1,\dots,\tilde{\alpha}^k)$ is trivializeable by a linear isomorphism on $\Phi(U')$ because it is constant, which immediately implies $\alpha = \Phi^*\tilde\alpha$ is trivializeable by continuous piecewise-$C^1$ diffeomorphisms.
	\end{proof}

	Before stating the main theorems, it is helpful to consider a necessary condition.

	\begin{proposition}
	Suppose $D = \ker \alpha$ is completely integrable by continuous piecewise-$C^1$ diffeomorphisms at $x_0$. Then for each polytope $\Delta$ containing $x_0$, and each pair of top-dimensional polytopes $T,T'$ both containing $\Delta$ as a face, then on an open set $U$ containing $x_0$, $i_\Delta^* \alpha_T = H_{T,T'} \cdot i_\Delta^* \alpha_{T'}$ for a unique smooth map $H_{T,T'}: \Delta \cap U \to GL(k)$.
	\end{proposition}
	\begin{proof}
	Observe that if $\Phi: U \to \mathbb{R}^n$ is continuous and piecewise-$C^1$, $d\Phi$ is single-valued on $T\Delta$, so $D \cap T\Delta = \ker i_\Delta^* \alpha$ is single-valued on $U \cap \Delta$. The proposition now follows from a basic linear algebra argument: if $V_1,V_2$ are two finite-dimensional vector fields and  $A_1,A_2: V_1 \to V_2$ are two surjective linear operators, such that $\ker A_1 = \ker A_2$, then there exists a unique operator $B \in \mathrm{Aut}(V_2)$ such that $A_1 = B A_2$.

	$B$ is simply $A_1 A_2^\dagger$, where $A_2^\dagger$ is any right-inverse of $A_2$. Let $W = A_2^\dagger(V_2)$, then every vector $v \in V_1$ can be written uniqely as $A_2^\dagger(v') + v''$, where $v' \in V_2$ and $v'' \in \ker A_1 = \ker A_2$. Then $A_1(v) = A_1A_2^\dagger(v')$, and $B A_2(v) =  A_1A_2^\dagger A_2 A_2^\dagger(v') = A_1A_2^\dagger(v')$. $B$ has full rank because $A_1$ and $A_2$ have full rank, and $B$ is unique because if $(B - B')A_2 = 0$ then $B^{-1}B' A_2 = A_2$, which implies $B^{-1}B' = I$ since $A_2$ is surjective.

	Lastly we simply note that $H_{T,T'}$ must be continuously differentiable, because $i_\Delta^*\alpha_T, i_\Delta^*\alpha_{T'}$ are both $C^1$ on $U \cap \Delta$.
	\end{proof}
	The integrability conditions presented below are both really sufficient conditions for there to exist a piecewise-$C^1$ map $H: U \to GL(k)$ such that $(H \cdot \alpha)^i = d\Phi^i$ for some functions $\Phi^1,\dots,\Phi^k$, and the functions $\Phi^i$ satisfy the conditions of the above lemma.

	For the first set of sufficient conditions, we require that $H_{T,T'} = I$ for each $T,T'$.

	\section{The First Set of Conditions}

	\begin{theorem}\label{mainthrm1}
		Suppose the following conditions hold:
		\begin{enumerate}
			\item There exists an open neighborhood $U$ of $x_0$ such that, for each $T \ni x_0$, $\alpha_T$ can be extended to a $C^1$ Pfaffian system on $U$, and there exists a smooth map $A_T: U \to \mathfrak{gl}(k)$ such that $d\alpha_T^i = -{A_T}^i_j \wedge \alpha_T^j$ for all $i$ in $1,\dots,k$ and $d{A_T}^i_j + {A_T}^i_l \wedge {A_T}^l_j = 0$ for all $i,j$ in $1,\dots,k$, \label{cond1}\\
			\item if $\Delta$ any polytope containing $x_0$ and $T,T'$ both contain $\Delta$ as a face, then $i_{\Delta}^*\alpha_T = i_{\Delta}^*\alpha_{T'}$ and $i_{\Delta}^*A_T = i_{\Delta}^*A_{T'}$ on $U \cap \Delta$, \label{cond2}\\
			\item and there exists a $k$-dimensional subspace $V \subset T_{x_0}M$ such that for any set of non-negative weights $\{c^T\}_{T \ni x_0}$ satisfying $\sum_T c^T = 1$, $\sum_T c^T \alpha_T|_{x_0}$ has full rank when restricted to $V$. \label{cond3}
		\end{enumerate}
	\end{theorem}

	Note that condition (\ref{cond1}) is not equivalent to the statement that $D$ is involutive on $T \cap U$. Involutivity on $T \cap U$ is only enough to guarantee complete integrability on $\mathring{T} \cap U$, so we require that $D|_T$ to be a restriction of an involutive distribution to $T \cap U$. Note also that condition (\ref{cond1}) is strictly stronger than involutivity in another way, because we require that $A_T$ is smooth and satisfies the structure equations of $GL(k)$. This is essentially an artifact of the proof, but taken together with (\ref{cond2}), this condition means we are only detecting Pfaffian systems of the form $\alpha^i = (H^{-1} \cdot \tilde{\alpha})^i$, where $\tilde{\alpha}$ is trivializeable by continuous piecewise-$C^1$ diffeomorphisms and the map $H: U \to GL(k)$ is continuous piecewise-smooth, while complete integrability is satisfied if $H$ is merely piecewise-$C^1$. Note that for piecewise-closed forms, $A_T = 0$ is always a valid choice and condition (\ref{cond1}) is always satisfied.

	\begin{proof}[proof of Theorem \ref{mainthrm1}]
		Let $A_T$ be the $\mathfrak{gl}(k)$-valued one-form assumed to exist by condition (\ref{cond1}). Since $A_T$ satisfies $dA_T + A_T \wedge A_T = 0$, the Fundamental Theorem of Calculus for Lie groups (see \cite[pp. 116]{Sharpe2000}), itself an application of the smooth Frobenius theorem, implies that $A_T = H_T^{-1}dH_T$ for a unique smooth map $H_T: U \to GL(k)$ such that $H_T(x_0) = I$. Since $i_{\Delta}^*A$ is single-valued for any shared face $\Delta = T \cap T'$ by condition (\ref{cond2}), it is also true that $H_T(x) = H_{T'}(x)$ for any $x \in T \cap T'$, so the collection of $H_T$'s fit together to a continuous piecewise-smooth map $H$.

		Next observe that by condition (\ref{cond1}), for each $i$ and each $T$,

		\[d(H_T \cdot \alpha_T)^i = d{H_T}^i_j \wedge \alpha_T^j - {H_T}^i_j ({H_T}^{-1}d{H_T})^j_l \wedge \alpha_T^l = 0,\]
		in other words $(H_T \cdot \alpha_T)^i$ is a closed differential one-form. Because $U$ is simply connected, this implies by de Rham's theorem that $(H_T \cdot \alpha_T)^i$ is an exact one-form on $U$, so there exists a map $\Phi_T^i: U \to \mathbb{R}^n$ such that $d\Phi_T^i = (H_T \cdot \alpha_T)^i$.  		
		Suppose $\Delta$ is a shared face of $T$ and $T'$ containing $x_0$, and let $x \in \Delta$. Then for any smooth path $\sigma: [0,1] \to \Delta$ such that $\sigma(0) = x_0$ and $\sigma(1) = x$, 
		\begin{align*}\Phi_T^i(x) - \Phi_{T'}^i(x) &= \int_0^1 \left[d\Phi^i_T(\dot{\sigma}(t)) - d\Phi^i_{T'}(\dot{\sigma}(t))\right]dt \\
		&= \int_0^1 \left[H^i_j(\sigma(t))\alpha^j_T(\dot{\sigma}(t)) - H^i_j(\sigma(t))\alpha^j_{T'}(\dot{\sigma}(t))\right]dt\\
		&= 0,\end{align*}
		where the last equality is true by condition (\ref{cond2}) and the fact that $H$ is continuous. Therefore, $\Phi_T^i(x) = \Phi_{T'}^i(x)$ for each $x \in \Delta \cap U$ and $i = 1,\dots,k$, so $\Phi$ is continuous (and hence Lipschitz) on $U$. Let $\bar\alpha_T := H^{-1} \cdot \alpha_T$. Then $\bar\alpha_T^i = d\Phi_T^i$, and since $\bar\alpha_T^i|_{x_0} = \alpha_T^i|_{x_0}$, condition (\ref{cond3}) is also satisfied for $\bar\alpha$. Therefore by Lemma (\ref{mainlemma}), $\bar\alpha$ is trivializeable by continuous peicewise-$C^1$ diffeomorphisms at $x_0$, implying that $\alpha$ is completely integrable by continuous piecewise-$C^1$ diffeomorphisms at $x_0$.
	\end{proof}
	
	\begin{example}\label{example1}
		Let $M = \mathbb{R}^2$ and let $\mathcal{T}$ be any mesh such that $0$ is the intersection of four polytopes $T_1,\dots,T_4$, which in a small enough neighborhood around $0$ segment the plane in to quadrants, so $T_1$ is the first quadrant, $T_2$ is the second quadrant, etc. Set $\alpha_{T_1} = -dx + dy$, $\alpha_{T_2} = dx + dy$, $\alpha_{T_3} = dx - dy$, and $\alpha_{T_4} = - dx - dy$. If $\Phi$ exists such that $\alpha_T = \Phi_T^*dx$ around $0$ and $\Phi(0) = 0$, then the sets $\{(x,x) : x \in \mathbb{R}\}$ and $\{(x,-x) : x \in \mathbb{R}\}$ are both tangent to $\alpha$ and contain $0$, so they must both be mapped by $\Phi$ continuously onto the set $\{(0,x) : |x| < \epsilon\}$ for sufficiently small $\epsilon > 0$. This contradicts the assertion that $\Phi$ is one-to-one.
		
		\begin{figure}
			\centering
			\includegraphics[width=0.7\linewidth]{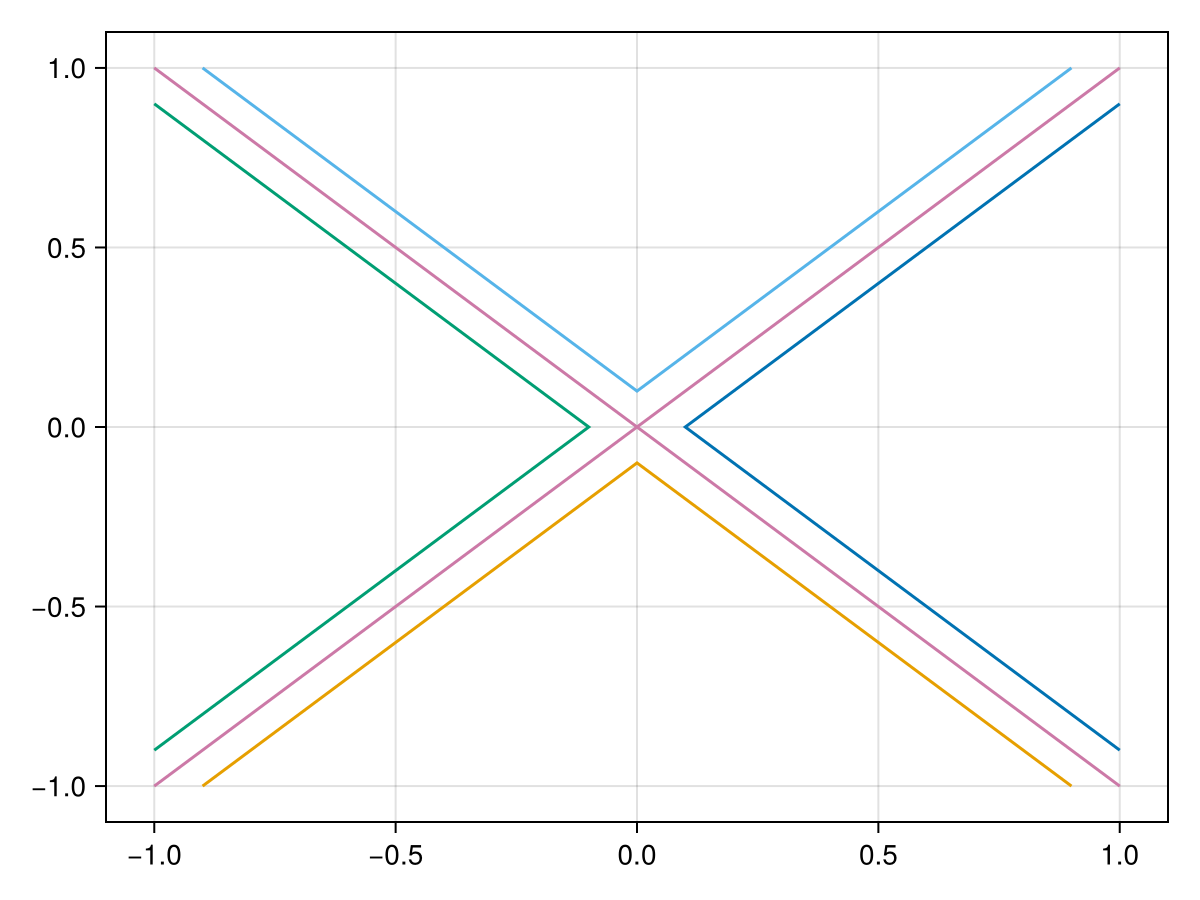}
			\caption{Graph of some of the "integral manifolds" of a distribution satisfying conditions (\ref{cond1}) and (\ref{cond2}) of Theorem \ref{mainthrm1}, but failing condition (\ref{cond3}). One of these "integral manifolds" is not a PL manifold.}
		\end{figure}
	\end{example}
	
	This example, and many like it, essentially stems from the fact that continuity is not sufficient for an ordinary differential equation to have unique solutions. The inverse function theorem as applied in Theorem \ref{mainthrm1} resolves the question.

	\section{The Second Set of Conditions}

	Next we turn to a more topological condition. The idea here is to find a special transversal to $D$ that gives enough information to construct the maps $\Phi^i$ even without the ability to construct a piecewise-smooth matrix field $H$ like in Theorem \ref{mainthrm1}.

	\begin{definition}
	A \emph{mesh transversal} to $\alpha$ through $x_0$ is a $C^1$ embedding $\gamma: [-1,1]^k \to M$ such that $\gamma(0) = x_0$, $\gamma^* \alpha_T$ has full rank for all $T$, and for each top-dimensional polytope $T$ containing $x_0$, the image of $\gamma$ intersects $T$ on a open subset of a codimension-$\le k$ face of $T$.
	\end{definition}

	The existence of a mesh transversal through $x_0$ is fairly strong. For instance, if $x_0$ lies in the interior of a polytope $\Delta_d$, where $d \le k$, then it rules out the possibility that $D|_{x_0}$ is tangent to $T_{x_0}\Delta_d$. It can also restrict the geometry of the mesh itself; for instance if $x_0$ is itself a codimension-$n$ polytope of the mesh, i.e. a vertex shared by several polytopes, there may not exist any smooth paths through that vertex that lie in the $n-1$-skeleten (the union of all polytopes of codimension $> 0$) of the mesh.
	
	\begin{theorem}\label{mainthrm2}
		Suppose the following conditions hold:
		\begin{enumerate}
			\item There exists an open neighborhood $U$ of $x_0$ such that, for each $T \ni x_0$, $\alpha_T$ can be extended to a $C^1$ Pfaffian system on $U$, and there exists a continuous map $A_T: U \to \mathfrak{gl}(k)$ such that $d\alpha_T^i = -{A_T}^i_j \wedge \alpha_T^j$ for all $i$ in $1,\dots,k$\label{cond4}\\
			\item if $\Delta$ any polytope containing $x_0$ and $T,T'$ both contain $\Delta$ as a face, then there exists a map $H_{T,T'}: U \cap \Delta \to GL(k)$ such that $i_{\Delta}^*\alpha_T = H_{T,T'} \cdot i_{\Delta}^*\alpha_{T'}$ on $U \cap \Delta$, \label{cond5}\\
			\item and there exists a mesh transversal to $\alpha$ through $x_0$. \label{cond6}
	\end{enumerate}
		Then $D = \ker \alpha$ is completely integrable by continuous piecewise-$C^1$ diffeomorphisms.
	\end{theorem}
	
	\begin{proof}
		By condition (\ref{cond4}) and the smooth Frobenius theorem, $\alpha_T$ is completely integrable on $U \cap T$. Let $\Phi_T: U \cap T \to \mathbb{R}^n$ be $C^1$ embeddings such that $d\Phi_T(D) = \mathrm{Span}(\pdiff{}{y^{k+1}},\dots,\pdiff{}{y^n})$. Let $S_T := \Phi_T(\gamma([-1,1]^k))$. Since $\gamma$ is a transversal, $S_T$ meets each level set $\{u\} \times \mathbb{R}^{n-k}$ (where $u \in \mathbb{R}^k$ and $\|u\| < \epsilon$ for some small-enough $\epsilon$) at a unique point $z_u \in S_T$. Let $\Gamma_T: B_\epsilon(0) \to S_T$ be the map $u \mapsto z_u$. This is a $C^1$ embedding because $\gamma$ is a $C^1$ embedding and a transversal, and it satisfies $\Gamma_T(0) = 0$.

		Let $\Gamma_T(u) := (\Gamma_T^1(u),\Gamma_T^2(u))$, so $\Gamma_T^1(u)$ is the first $k$ components of $\Gamma_T(u)$ and $\Gamma_T^2(u)$ is the last $n-k$ components. Then we define a new map

		\[\Psi_T(u,v) := (d\gamma^{-1}\circ \Phi_T^{-1} \circ \Gamma_T(u),v - \Gamma_T^2(u)).\]

		$\Psi_T$ is a $C^1$ embedding $B_\epsilon(0) \to \mathbb{R}^n$ which maps $0 \mapsto 0$, and it satisfies $\Psi_T \circ \Phi_T \circ \gamma(s) = (s,0)$ for all $s \in \gamma^{-1}\circ \Phi_T^{-1}(B_\epsilon(0))$. Let ${\Phi_T^i}' := \Psi_T^i \circ \Phi_T$, and let ${\Phi^i}'$ be the piecewise-$C^1$ maps such that ${\Phi^i}'|_T = {\Phi^i_T}'$.

		Let us show that the maps ${\Phi^i}'$ are continuous on a small enough open subset $U'' \subset U'$ such that $x_0 \in U''$. Let $\Delta$ be any polytope containing $x_0$ which lies at the intersection of $T$ and $T'$. By condition (\ref{cond5}), $D_{\Delta} := D \cap T\Delta$ is single-valued and $C^1$ on $U' \cap \Delta$, and by condition (\ref{cond4}) it is also completely integrable on $U' \cap \Delta$. Therefore, for a small enough open subset $U'' \subset U'$ containing $x_0$, $\gamma|_{\gamma^{-1}(\Delta)}$ is transverse to $D_{\Delta}$, and $U''$ can be possibly shrunken so $\gamma^{-1}(U'' \cap \Delta')$ is also simply connected. Let $U''$ be small enough that this is true for all $\Delta$ containing $x_0$.

		Then for all $x \in U'' \cap \Delta$, there exists a unique integral curve $\sigma$ of $D_{\Delta}$ such that $\sigma(0) = x$ and $\sigma(1) = \gamma(p)$ for a unique point $p \in \gamma^{-1}(\Delta)$. By the fundamental theorem of calculus, we have

		\[\Phi_T^i(x) - \Phi_{T'}^i(x) = \Phi_{T}^i(p) - \Phi_{T'}^i(p) - \int_0^1 d\Phi_T^i(\dot{\sigma}(t)) - d\Phi_{T'}^i(\dot{\sigma}(t)) dt.\] 

		Since $\sigma$ is an integral curve of $D_{\Delta}$, $d\Phi_T^i(\dot{\sigma})$ and $d\Phi_{T'}^i(\dot{\sigma})$ both vanish. Meanwhile, because $\Phi_T^i \circ \gamma(s) = s^i$ for $i = 1,\dots,k$, $\Phi^i$ is single-valued on $\gamma([-1,1]^k)$, which contains $p$. Therefore, $\Phi_T^i(x) = \Phi^i_{T'}(x)$. Note that $d\Phi^i$ is also continuous when restricted to the $k$-dimensional subspace $V = d\gamma(T_{0}\mathbb{R}^k) \subset T_{x_0}M$. 

		So, setting $\bar{\alpha}^i := d\Phi^i$, Lemma \ref{mainlemma} implies that $\bar{\alpha}$ is trivializeable by piecewise-smooth diffeomorphisms. And since $\ker\bar\alpha = D$, we have that $D$ is completely integrable by piecewise-smooth diffeomorphisms at $x_0$.
	\end{proof}

	\printbibliography

\end{document}